\documentclass[final,leqno]{siamltex}
\usepackage{amsmath}
\usepackage{amsfonts}
\usepackage{color}
\usepackage[latin1]{inputenc}
\usepackage{algpseudocode}
\usepackage{algorithm}
\usepackage{multirow}
\usepackage{lscape}
\usepackage[update,prepend]{epstopdf} 
\usepackage{graphicx}
\usepackage{multirow}
\usepackage{url}
\usepackage{diagbox}

\newcommand{\be}{\begin{eqnarray}}
\newcommand{\e}{\end{eqnarray}}
\newcommand{\bes}{\begin{eqnarray*}}
\newcommand{\es}{\end{eqnarray*}}
\newcommand{\beq}{\begin{equation}}
\newcommand{\eeq}{\end{equation}}

\newtheorem{remark}{Remark}
\newtheorem{assumption}{Assumption}
\author{Peter Benner, Akwum Onwunta \thanks{Corresponding author}, and Martin Stoll\\
Max Planck Institute for Dynamics of Complex Technical Systems, Sandtorstrasse 1, 39106 Magdeburg, Germany }

\begin{document}
\title{On the existence and uniqueness of the solution of a parabolic optimal
control problem with uncertain inputs
\thanks{This work is partially funded by the U. S. Department of Energy Office of Advanced Scientific Computing Research, 
Applied Mathematics program, under Award Number DE-SC0009301.}}
\author{Peter Benner\footnotemark[1], 
Akwum Onwunta\footnotemark[2] 
\and Martin Stoll\footnotemark[3]}
\renewcommand{\thefootnote}{\fnsymbol{footnote}}
\footnotetext[1]{Computational Methods in Systems and Control Theory, Max Planck Institute
for Dynamics of Complex Technical Systems,
Sandtorstrasse 1,
39106 Magdeburg,
Germany,
(\url{benner@mpi-magdeburg.mpg.de})}
\footnotetext[2]{Corresponding author; Department of Computer Science,
University of Maryland, College Park,
MD 20742, USA,
(\url{onwunta@cs.umd.edu}).}
\footnotetext[3]{
Technische Universit\"{a}t Chemnitz, Faculty of Mathematics,
Professorship Scientific Computing, 09107 Chemnitz, Germany,
(\url{martin.stoll@mathematik.tu-chemnitz.de}).}
\renewcommand{\thefootnote}{\arabic{footnote}}
\maketitle

\begin{abstract}
In this note, we consider the existence and uniqueness of the solution of a time-dependent optimal control problem 
constrained by a partial differential equation with uncertain inputs. 
Relying on the Lions' Lemma for deterministic  problems, we furthermore characterize the optimal control of the
stochastic problem.
\end{abstract}



\pagestyle{myheadings}
\thispagestyle{plain}
\markboth{Existence and uniqueness results for stochastic optimal control problems}{Existence and uniqueness results for stochastic optimal control problems}
\section{Introduction}
\label{sec1}

This work is essentially concerned with the theoretical analysis of the existence and uniqueness of a parabolic stochastic  optimal control problem 
(SOCP) considered in an earlier paper \cite{BOS15a} by the authors.
In \cite{BOS15a}, we propose efficient low-rank iterative solvers 
 for solving the linear systems resulting from the stochastic Galerkin
 finite element method discretization of the SOCP. 
From a computational point of view, the work of Rosseel and Wells in \cite{EW12} 
is directly related to the problem considered herein. However, the paper \cite{EW12} treats only stationary SOCP and, as \cite{BOS15a},
does not deal with the theoretical issues
of the existence and uniqueness of the solution of the considered problem.
So, one goal of this paper is to close this apparent gap in the literature 
regarding the  existence and uniqueness of SOCPs.
Unlike other related literature on SOCPs
(see e.g. \cite{HLM11} and the references therein), a special 
feature of the parabolic SOCP considered herein (as well as in  \cite{BOS15a, AO16, EW12})
is the presence of the  standard deviation
of the state variable as a risk measure in the cost functional.

In this work, we  first establish the  existence and uniqueness of the solution of the 
 parabolic SOCP. Next, we 
rely mainly on the Lions' Lemma \cite{Lions71} to characterize the optimal control of the stochastic problem.
 However, we proceed first to  Section \ref{probst} to provide the mathematical setting for the considered problem and recall Lions'
Lemma. The main results are stated and proved in Section \ref{result}.

\section{Problem statement}
\label{probst}
 In this note, we study the existence and uniqueness of the solution of a parabolic optimal control problem with stochastic inputs (SOCP). 
Before we proceed to state our problem, we first fix some notation that we will use in the sequel. To that end, 
let $\mathcal{D}\subset \mathbb{R}^d$ with $d\in\{1,2,3\}$ be a domain with Lipschitz boundary
$\partial \mathcal{D}.$ Moreover,  for $T>0,$ we denote the time interval by $[0,T].$ We recall that by  a random field $z: \mathcal{D}\times \Omega \rightarrow \mathbb{R},$
 we mean that $z(x,\cdot)$ is a random variable defined on  the complete probability space $(\Omega,\mathcal{F},{\mathbb{P}})$ for each  $x\in \mathcal{D}.$
Here, $\Omega$ is the set of outcomes, $\mathcal{F}\subset 2^{\Omega}$ is the $\sigma$-algebra of events, and $\mathbb{P}:\mathcal{F}\rightarrow [0,1]$ is an appropriate
probability measure. We consider only time-dependent random fields and assume that they are in the  tensor product Hilbert space $L^2(0,T;L^2(\mathcal{D}))\otimes L^2(\Omega)$ which is endowed with the norm
\[
||\upsilon||_{L^2(0,T;\mathcal{D})\otimes L^2(\Omega)}:=\left(\int_\Omega ||\upsilon(\cdot,\cdot,\omega)||^2_{L^2(0,T;L^2(\mathcal{D}))} \;d\mathbb{P}(\omega)\right)^{\frac{1}{2}} <\infty, 
\]
where ${L}^2(\Omega):={L}^2(\Omega,\mathcal{F},{\mathbb{P}}).$
 For any random variable $z$ defined on  $(\Omega,\mathcal{F},{\mathbb{P}}),$
the standard deviation $\mbox{std}(z)$  and the mean $\mathbb{E}(z)$ of $z$ are given, respectively,  by
\be
\label{stdy}
\mbox{std}(z) = \left[\int_\Omega (z - \mathbb{E}(z))^2 \;d\mathbb{P}(\omega)\right]^{\frac{1}{2}} \;\;\mbox{and}\;\;
\mathbb{E}(z)=\int_\Omega z\;d\mathbb{P}(\omega)<\infty.
\e
It is pertinent here to recall also that  the variance $\mbox{var}(z)$ of $z$ is given by
\be
\label{varnew}
\mbox{var}(z)= (\mbox{std}(z))^2 = \mathbb{E}(z^2) - (\mathbb{E}(z))^2.
\e

In what follows, we write $\mathbb{P}$-a.s to mean  $\mathbb{P}$-almost surely.
Next, we set    
$\mathcal{X}:= L^2(0,T;H^1_0(\mathcal{D}))\otimes L^2(\Omega) $   and  let
$\partial_t y\in L^2(0,T;H^{-1}(\mathcal{D}))\otimes L^2(\Omega).$ Also, let
$\mathcal{U}:= L^2(0,T;L^2(\mathcal{D}))\otimes L^2(\Omega)  $ 
be the control space and set $\mathcal{Y}=\mathcal{U},$ where $\mathcal{Y}$ is the state space.
Moreover, let $\mathcal{W}:= L^2(0,T; L^2(\mathcal{D})).$
 Note then that $\mathcal{W}, \; \mathcal{X} \subset L^2(0,T;L^2(\mathcal{D}))\otimes L^2(\Omega) =\mathcal{Y}.$
Finally, we set the Hilbert space
 $\mathcal{V}:=L^2(\mathcal{D})\otimes L^2(\Omega)$ and let $\mathcal{V}^{\prime}$ be the dual of 
$\mathcal{V}.$

 In  this work,  we shall focus on distributed control problems, although we believe that our discussion generalizes to 
boundary control problems. More explicitly,  we formulate our model problems as: 
\begin{eqnarray}
\min_{u\in  \mathcal{U}_{ad}}\mathcal{J}(y,u) &:=& \frac{1}{2}||y- y_d||^2_{\mathcal{Y}}
+ \frac{\alpha}{2}||\mbox{std}(y)||^2_{\mathcal{W}}
+ \frac{\beta}{2}||u||^2_{\mathcal{U}} \label{Junst}
\end{eqnarray}
subject, $\mathbb{P}$-a.s, to
 \be
 \label{parmodelprob}
\left\{
\begin{aligned}
    \frac{\partial y(t,x,\omega)}{\partial  t}  + A(x,\omega)y(t,x,\omega)) &= u(t,x,\omega),
       \;\;\mbox{in} \; \; (0,T]\times \mathcal{D}\times  \Omega,  \quad\\
    y(t,x,\omega) &= 0, \;\;\mbox{on} \;\; (0,T]\times \partial\mathcal{D}\times   \Omega, \quad\\
     y(0,x,\omega) &= 0, \;\;\mbox{in} \; \; \mathcal{D}\times\Omega,
\end{aligned}
\right.
\e
where $A:\mathcal{V} \rightarrow \mathcal{V}^{\prime}$ is a linear operator
that contains some random parameters.
Moreover, $\mathcal{J}$ is a cost functional of tracking-type, which
includes a risk penalization via the standard deviation. The functions  $y,u$  and $y_d$ are, in general, real-valued  functions 
representing, respectively, the state,  the  control and
the prescribed target system response (or desired state). 
Without loss of generality, we assume that  the state  $y\in \mathcal{Y}$ and the  control  function $u\in\mathcal{U}$ are random fields 
while the desired state  $y_d \in \mathcal{Y}$ is modeled deterministically.
The  constant  $\beta > 0$  in (\ref{Junst})
 represents the parameter for the penalization of the action of the control 
 $u,$ whereas  $\alpha\geq 0$ is the so-called {\it risk-aversion} parameter that penalizes the standard deviation $\mbox{std}(y)$ of the state  $y.$  
 The objective functional $\mathcal{J} $ is a
 deterministic quantity with uncertain terms.  The set   $  \mathcal{U}_{ad}$ is the so-called convex admissible set, and is given by
 \be
 \label{uad}
  \mathcal{U}_{ad}:=  \{ u\in \mathcal{U}:  u(t,x,\omega)\geq 0 \;\; \mathbb{P}{\mbox{-a.s}} \;\; {\mbox{in}} \;\; [0,T]\times \mathcal{D}\times  \Omega\}.
 \e
 We shall need the following assumptions on $A$ in the sequel.
 
\newpage
\begin{assumption}
 \label{assmp1}
 \begin{enumerate}
 \item [(a)] $A$ is coercive: there  exist  constants $c>0$ and $\theta$ such that, $\mathbb{P}$-a.s,
 \[
 \left(A  v, v\right) + \theta ||v||^2_\mathcal{H} \geq c ||v||_\mathcal{V}, \;\; \forall v\in \mathcal{V}, 
 \]
 where  the space $\mathcal{H}$ is chosen such that $\mathcal{V}\subset \mathcal{H} \subset \mathcal{V}^{\prime}$ and
$\mathcal{V}$ is dense in $\mathcal{H}.$
  \item [(b)] $A$ is self-adjoint: $\left(A u, v\right) = \left( u, A^{\star}v\right), \;\;\; \forall u,v\in \mathcal{V},$ \;  $\mathbb{P}$-a.s.
\end{enumerate}
\end{assumption}

A prominent example of the operator $A$ is the diffusion operator considered, for instance, in 
\cite{BOS15a, EW12}:
\be
\label{diffeq}
A:= -\nabla\cdot a(x,\omega)\nabla,  
\e
in which case we assume that the random field  $a(x,\omega)$ is
uniformly positive in $\mathcal{D}\times\Omega.$ That is,
 there  exist strictly positive constants $a_{\min}$ and $a_{\max},$ with
 $a_{\min}\leq a_{\max},$ such that
\[
\mathbb{P}\left( \omega\in \Omega: a(x,\omega)\in [a_{\min}, a_{\max}],\, \forall x\in\mathcal{D}  \right)=1.
\]
The weak formulation of the SOCP  (\ref{Junst}) and (\ref{parmodelprob})
above is  given by
%

 \begin{eqnarray}
 \min_{u\in   \mathcal{U}_{ad}}\mathcal{J}(u)
  &=&\frac{1}{2}\mathbb{E}\int_0^T\int_{\mathcal{D}} (y(u)- y_d)^2 \;dxdt 
   +  \frac{\alpha}{2} \mathbb{E}\int_0^T\int_{\mathcal{D}}\left[ y(u) - (\mathbb{E} y(u) ) \right]^2 \;dxdt  \nonumber\\
  && +  \frac{\beta}{2} \mathbb{E}\int_0^T\int_{\mathcal{D}} u^2  \;dxdt  \label{varj}
\end{eqnarray}
subject, $\mathbb{P}$-a.s, to
 \begin{eqnarray}
\label{varformula}
\mathbb{E}\int_0^T\int_{\mathcal{D}}  \partial_t y v \;dxdt + \mathcal{B}(y,v) &=&
 \mathbb{E}\int_0^T\int_{\mathcal{D}} u v \;dxdt, \;\;\; v\in H^1_0(\mathcal{D})\otimes L^2(\Omega),
\end{eqnarray}
where $\mathcal{B}$ is a bilinear form  of the operator $A$ defined on the 
tensor product space $H^1_0(\mathcal{D})\otimes L^2(\Omega)$.
In particular, if $A= -\nabla\cdot a(x,\omega)\nabla,$ then
\[
 \mathcal{B}(y,v) := \mathbb{E}\int_0^T\int_{\mathcal{D}} a \nabla y\cdot \nabla v \;dxdt.
\]

Next, we proceed to Section \ref{result} to establish our existence and uniqueness results for 
the parabolic SOCP (\ref{Junst}) -- (\ref{parmodelprob}).

\section{Existence and uniqueness results}
\label{result}
In our subsequent discussion,  we shall rely 
explicitly on the following definition. 
\begin{definition}
 A function $\bar{u}\in \mathcal{U}_{ad}$ is called an optimal control and 
  ${\bar{y}} = y(\bar{u})$ the associated optimal state corresponding to the 
the SOCP  (\ref{Junst}) and (\ref{parmodelprob}) if, $\mathbb{P}$-a.s, the following
expression holds:
\[
\mathcal{J}(\bar{y},\bar{u}) \leq \mathcal{J}(y(u),u), \quad \;\; \forall u\in \mathcal{U}_{ad}. 
\]
\end{definition}


We can now state the following result.
\begin{theorem}
\label{exun1}
Let $\{ \mathcal{H}_1, ||\cdot|| \}$ and $\{ \mathcal{H}_2, ||\cdot|| \}$ be Hilbert spaces. Suppose that 
$\tilde{\mathcal{H}}_1\subset \mathcal{H}_1$ is a non-empty, closed and convex set. Let
 $y_d\in \mathcal{H}_1$  and the  constants $\gamma, \eta \geq 0$ be given.
Furthermore,
let $S: \mathcal{H}_1 \mapsto\mathcal{H}_2$ be a continuous linear operator. Then,
the quadratic  Hilbert space optimization problem
\[
 \min_{u\in  \tilde{\mathcal{H}}_1 } f(u) = \frac{1}{2}||Su - y_d||^2_{\mathcal{H}_2}
 + \frac{\gamma}{2}||\mbox{std}(Su)||^2_{\mathcal{H}_2}
 + \frac{\eta}{2}||u||^2_{\mathcal{H}_1} 
\]
admits, $\mathbb{P}$-a.s,  an optimal solution ${\bar{u}}\in \tilde{\mathcal{H}}_1$. If $\eta >0$, 
then ${\bar{u}}$ is uniquely determined.
\end{theorem}
\begin{proof}
 Note first that the function $f(u)\geq 0$ is continuous and convex. The proof of Theorem \ref{exun1}
 therefore follows analogously to that of \cite[Theorem 2.14]{FT10}.
\end{proof}

Suppose now that we set $\tilde{\mathcal{H}}_1 =  \mathcal{U}_{ad}, \;\; \mathcal{H}_1= \mathcal{U}$
and $\mathcal{H}_2 = \mathcal{Y}.$  
Observe then that, by \cite[Theorems 3.12, 3.13]{FT10}, for any $u\in \mathcal{U},$ there exists 
a unique solution to the parabolic initial-boundary value problem (\ref{parmodelprob}).
Now, let the mapping 
\[
G_{\mathcal{U}}: \mathcal{U} \mapsto \mathcal{X} \subset \mathcal{Y},\; u \mapsto y(u) 
\]
be the so-called 
control-to-state operator. Observe then  that $G_{\mathcal{U}}$ is a continuous 
linear operator \cite[pp. 50]{FT10}.  Moreover, let $E_{\mathcal{Y}}: \mathcal{X} \mapsto \mathcal{Y}$ denote the 
embedding operator that assigns to each $y\in  \mathcal{X}\subset \mathcal{Y}$ the same function in 
$\mathcal{Y} .$ Note that $E_{\mathcal{Y}}$ is also linear and continuous. Thus,
the composition  $ u \mapsto y(u) \mapsto  y(u) $
is a continuous linear operator $S:=E_{\mathcal{Y}}G_{\mathcal{U}}: u \mapsto y(u). $ 
Hence, it turns out that if we substitute $S$ into the cost functional $\mathcal{J}(y,u)$ in
(\ref{Junst}), then we eliminate the parabolic initial-boundary value problem to arrive at the
following quadratic minimization  problem in the Hilbert space $\mathcal{U}$:
 \begin{eqnarray}
 \min_{u\in   \mathcal{U}_{ad}}\mathcal{J}(u)
  &=&\frac{1}{2}\mathbb{E}\int_0^T\int_{\mathcal{D}} (Su- y_d)^2 \;dxdt 
   +  \frac{\alpha}{2} \mathbb{E}\int_0^T\int_{\mathcal{D}}\left[ Su - (\mathbb{E} Su ) \right]^2 \;dxdt  \nonumber\\
  && +  \frac{\beta}{2} \mathbb{E}\int_0^T\int_{\mathcal{D}} u^2  \;dxdt.  \label{varj2}
\end{eqnarray}

Furthermore, note  that  $\mathcal{U}_{ad}$ as defined by (\ref{uad}) is indeed  non-empty, closed and convex.
Thus,  we can infer from Theorem \ref{exun1} that there exists an optimal control ${\bar{u}}$ 
for the  problem (\ref{varj2}).
Moreover, by our initial assumption the regularization parameter $\beta >0.$
Hence, we have  indeed established the following result.
\begin{theorem}
\label{exun}
Under the assumptions of Theorem \ref{exun1}, the SOCP (\ref{Junst}) -- (\ref{parmodelprob})
has at least one optimal control ${\bar{u}}\in \mathcal{U}_{ad}$. If $\beta >0$, 
then ${\bar{u}}$ is uniquely determined.
\end{theorem}

Although Theorem \ref{exun} establishes the existence and uniqueness of a solution to
the optimal control problem (\ref{Junst}) -- (\ref{parmodelprob}),
it is not constructive in the sense that the
theorem provides no indication as to how this solution can be obtained. In the
 following, we will address this issue and provide a characterization of the
optimal control. To this end, 
a chief corner stone in our subsequent discussion in this contribution is the following 
result in the deterministic framework, which 
is often known as the Lions' lemma \cite[p.~10]{Lions71}.
\begin{theorem}\cite[Theorem 1.3]{Lions71}
\label{lions}
Suppose the cost functional $v\mapsto \mathcal{J}(v)$ is strictly convex and differentiable.
Then, there exists a unique optimal  control $\bar{u}\in  \mathcal{U}_{ad}$   if and only if
\be
\label{keyeqn}
\mathcal{J}^\prime (\bar{u}) \cdot (v - \bar{u}) \geq 0, \;\;\;\; \forall v\in  \mathcal{U}_{ad},
\e
where 
\be
\label{derivj}
 \mathcal{J}^\prime (\bar{u}) \cdot (w) := \lim_{h\rightarrow 0}\frac{\mathcal{J}(\bar{u}+hw)- \mathcal{J}(\bar{u})}{h},
\e
is the derivative of $\mathcal{J}$ with respect to $u$ in the direction of $w.$
\end{theorem}

Note that it is very easy to check that the cost functional $\mathcal{J}$ in (\ref{Junst}) is strictly convex.
We can now prove the following  characterization result for the optimal control.
\begin{theorem}
 The SOCP  given  by (\ref{Junst}) -- (\ref{parmodelprob}) has a unique solution $(y,u)$ if and only if there exists 
 a co-state variable $\lambda\in \mathcal{Y}$
such that the triplet $(y,u,\lambda)$ satisfies, $\mathbb{P}$-a.s , the following optimality system:
\begin{eqnarray}
 \frac{\partial y(u)}{\partial  t} + A  y(u)  &= & u ,\label{pde}\\
 y(u)\mid_{t=0 }&=& 0, \;\;\; y(u)\in \mathcal{Y}.\label{pde1}
\end{eqnarray}
\begin{eqnarray}
 -\frac{\partial \lambda(u)}{\partial t} + A^\ast   \lambda(u) 
 &=&  (1+\alpha)y(u) - \alpha \mathbb{E}( y(u))  - y_d,\label{starrnew}\\
 \lambda(u)\mid_{t=T }&=& 0, \;\;\; \lambda(u)\in \mathcal{Y}. \label{starrnew1}
\end{eqnarray}
\begin{eqnarray}
\mathbb{E}\int_0^T\int_{\mathcal{D}} (\lambda(u) + \beta u)\cdot (v-u) \;dxdt \geq 0, \;\;\; u,v\in \mathcal{U}_{ad}.
\end{eqnarray}
\end{theorem}
We note here that we have dropped the dependence on $(t,x,\omega)$ for notational convenience.
\begin{proof}
It suffices  to show that
 the condition (\ref{keyeqn}) in Theorem \ref{lions} holds. Now, using   (\ref{varnew}) and the fact that $y=y(u),$ note that (\ref{Junst}) can be re-written as
 \begin{eqnarray}
  \mathcal{J}(u) &:=& \mathcal{J}_1(u) + \mathcal{J}_2(u)  - \mathcal{J}_3(u)  +\mathcal{J}_4(u) \nonumber\\
  &=&\frac{1}{2}\mathbb{E}\int_0^T\int_{\mathcal{D}} (y(u)- y_d)^2 \;dxdt 
   +  \frac{\alpha}{2} \mathbb{E}\int_0^T\int_{\mathcal{D}} y(u)^2 \;dxdt  \nonumber\\
  &&- \frac{\alpha}{2}  \int_0^T\int_{\mathcal{D}} (\mathbb{E} y(u))^2 \;dxdt
   +  \frac{\beta}{2} \mathbb{E}\int_0^T\int_{\mathcal{D}} u^2  \;dxdt.  \label{newj}
\end{eqnarray}
Using the definition (\ref{derivj}), we find that 
\begin{eqnarray}
  \mathcal{J}_1^{\prime} (u) \cdot (v-u) &=&   \lim_{h\rightarrow 0}
  \frac{ \mathbb{E}\int_0^T\int_{\mathcal{D}} [y(u + h(v-u) ) - y_d]^2\;dxdt  - \mathbb{E}\int_0^T\int_{\mathcal{D}} [y(u) -y_d]^2\;dxdt}{2h}   \nonumber\\
  &=& \lim_{h\rightarrow 0}
  \frac{\mathbb{E}\int_0^T\int_{\mathcal{D}}y(u + h(v-u) )^2 - y(u)^2\;dxdt }{2h}\nonumber\\
  && - \lim_{h\rightarrow 0}  \frac{\mathbb{E}\int_0^T\int_{\mathcal{D}} 2y_d( y(u + h(v-u) ) - y(u))\;dxdt}{2h}   \nonumber\\
    &=& \mathbb{E}\int_0^T\int_{\mathcal{D}} y(u) y^{\prime}(u)\cdot (v-u) \;dxdt
         -   \mathbb{E}\int_0^T\int_{\mathcal{D}} y_d y^{\prime}(u)\cdot (v-u) \;dxdt \nonumber\\
   &=& \mathbb{E}\int_0^T\int_{\mathcal{D}} (y(u)- y_d) y^{\prime}(u)\cdot (v-u) \;dxdt, \nonumber
  \end{eqnarray}
and
\begin{eqnarray}
  \mathcal{J}_4^{\prime} (u) \cdot (v-u) &=& \beta\cdot  \lim_{h\rightarrow 0}
  \frac{ \mathbb{E}\int_0^T\int_{\mathcal{D}} (u + h(v-u) )^2\;dxdt  - \mathbb{E}\int_0^T\int_{\mathcal{D}} u^2\;dxdt}{2h}   \nonumber\\
  &=&\beta\cdot \lim_{h\rightarrow 0}
  \frac{\mathbb{E}\int_0^T\int_{\mathcal{D}}(u^2 + h^2(v-u)^2  + 2 hu(v-u) )\;dxdt - \mathbb{E}\int_0^T\int_{\mathcal{D}} u^2\;dxdt}{2h}   \nonumber\\
   &=&  \beta \mathbb{E}\int_0^T\int_{\mathcal{D}} u\cdot (v-u)  \;dxdt.  \nonumber
\end{eqnarray}
One can easily perform similar calculations with  the terms $\mathcal{J}_2 $ and $ \mathcal{J}_3$  in (\ref{newj}) to obtain (see e.g. \cite{EW12}), respectively,
\[
  \mathcal{J}_2^{\prime} (u) \cdot (v-u) = \alpha \mathbb{E}\int_0^T\int_{\mathcal{D}} y(u) y^{\prime}(u)\cdot (v-u) \;dxdt,
\]
and
\[
  \mathcal{J}_3^{\prime} (u) \cdot (v-u) = \alpha \mathbb{E}\int_0^T\int_{\mathcal{D}} (\mathbb{E} y(u)) y^{\prime}(u)\cdot (v-u) \;dxdt.
\]
Hence, we have
\begin{eqnarray}
  \mathcal{J}^{\prime} (u) \cdot (v-u) &=& \left[ \mathcal{J}_1^{\prime} (u) + \mathcal{J}_2^{\prime} (u) -\mathcal{J}_3^{\prime} (u)
  + \mathcal{J}_4^{\prime} (u) \right]\cdot (v-u) \nonumber\\
  &=& \mathbb{E}\int_0^T\int_{\mathcal{D}} (y(u)- y_d) y^{\prime}(u)\cdot (v-u) \;dxdt \nonumber\\
   && + \; \alpha \mathbb{E}\int_0^T\int_{\mathcal{D}} y(u) y^{\prime}(u)\cdot (v-u) \;dxdt  \nonumber\\
  && -\; \alpha \mathbb{E}\int_0^T\int_{\mathcal{D}} (\mathbb{E} y(u)) y^{\prime}(u)\cdot (v-u) \;dxdt \nonumber\\
   && +\;  \beta \mathbb{E}\int_0^T\int_{\mathcal{D}} u\cdot (v-u)  \;dxdt.  \label{juv}
\end{eqnarray}

Next, let $\mathcal{L}:=\frac{\partial }{\partial  t} + A.$ 
 Note then that  Assumption \ref{assmp1},  as given by $(a)$  and $(b)$ which are directly below equation (\ref{parmodelprob}), implies that  $\mathcal{L}$ is invertible  and, indeed, from (\ref{pde}) one then gets
\be
\label{oper}
 \mathcal{L} y(u) =u  \Longrightarrow y(u) =  \mathcal{L}^{-1}  u,
\e
so that, using (\ref{oper}), the quantity $ y^{\prime}(u)\cdot (v-u) $ appearing in the first three terms of the expression (\ref{juv}) now yields
\[
 y^{\prime}(u)\cdot (v-u)  =  \mathcal{L}^{-1}  (v-u) = \mathcal{L}^{-1} (v) - \mathcal{L}^{-1}  (u) = y(v) - y(u).
\]
Thus, we have from (\ref{juv}) that
\begin{eqnarray}
  \mathcal{J}^{\prime} (u) \cdot (v-u) &=& \Psi(\alpha)  + \beta \mathbb{E}\int_0^T\int_{\mathcal{D}} u\cdot (v-u)  \;dxdt,  \label{condj}
\end{eqnarray}
where 
\begin{eqnarray}
\Psi(\alpha)  &=&
(1+\alpha)\mathbb{E}\int_0^T\int_{\mathcal{D}} y(u)\cdot (y(v) - y(u)) \;dxdt \nonumber\\
  && -\; \alpha \mathbb{E}\int_0^T\int_{\mathcal{D}} \mathbb{E} (y(u)) \cdot (y(v) - y(u)) \;dxdt  \nonumber\\
 && -\; \mathbb{E}\int_0^T\int_{\mathcal{D}} y_d \cdot (y(v) - y(u)) \;dxdt.  \label{phi}
\end{eqnarray}
Observe, once again, that Lions' lemma demands that the expression (\ref{condj}) be non-negative to ensure 
the existence and uniqueness of the solution to the SOCP (\ref{Junst}) and (\ref{parmodelprob}).
To that end, we next introduce the adjoint state $\lambda(v)$ by
\begin{eqnarray}
 -\frac{\partial \lambda(v)}{\partial t} + A^\ast   \lambda(v) 
 &=&  (1+\alpha)y(v) - \alpha \mathbb{E}( y(v))  - y_d,\label{starr}\\
 \lambda(v)\mid_{T=0 }&=& 0, \;\;\; \lambda(v)\in \mathcal{Y}.
\end{eqnarray}
Now, set $v=u$ in (\ref{starr}) and   multiply both sides of the equation  by $y(v) - y(u)$. 
Observe first that, by taking the expectation of the resulting expression and integrating
it over $[0,T]$ and $\mathcal{D}$, these two operations essentially transform the right hand 
side of (\ref{starr}) to the expression $\Psi(\alpha)$ in (\ref{phi}). Moreover,
  the two terms
on the left hand side of (\ref{starr}) now read 
{\small 
\begin{eqnarray}
\mathbb{E}\int_0^T\int_{\mathcal{D}} -\frac{\partial \lambda(u)}{\partial t} \cdot (y(v) - y(u)) \;dxdt 
&=& \mathbb{E}\int_0^T\int_{\mathcal{D}}  \lambda(u) \cdot \left(\frac{\partial y(v)}{\partial t}  - \frac{\partial y(u)}{\partial t} \right) \;dxdt, \label{a}
\end{eqnarray}
}
{\small
\begin{eqnarray}
\mathbb{E}\int_0^T\int_{\mathcal{D}} A^\ast  \lambda(u) \cdot (y(v) - y(u)) \;dxdt 
&=& \mathbb{E}\int_0^T\int_{\mathcal{D}}  \lambda(u) \cdot (A   y(v) - A  y(u)) \;dxdt. \label{b}
\end{eqnarray}
}To obtain (\ref{b}), we have used the fact that the operator $A$ is self-adjoint. Furthermore, we have used integration by parts,
together with the conditions (\ref{pde1}) and (\ref{starrnew1}) to obtain (\ref{a}).
Thus,  summing up (\ref{a})  and (\ref{b}), one gets
\begin{eqnarray}
 \mathbb{E}\int_0^T\int_{\mathcal{D}} 
\left( \lambda(u) \cdot \left(\frac{\partial }{\partial t}  + A \right) (y(v) - y(u)) \right)\;  dxdt
 &=&  \mathbb{E}\int_0^T\int_{\mathcal{D}}  \lambda(u) \cdot (v - u) \; dxdt\nonumber \\
 &=&\Psi(\alpha), \label{rls}
\end{eqnarray}
where we have explicitly used (\ref{pde}) in the first line of (\ref{rls}).
Hence, it follows from (\ref{condj}), (\ref{phi}) and  Theorem \ref{lions} that 
\begin{eqnarray}
  \mathcal{J}^{\prime} (u) \cdot (v-u) &=& \Psi(\alpha)  + \mathbb{E}\int_0^T\int_{\mathcal{D}} \beta  u\cdot (v-u)  \;dxdt,  \nonumber\\
  &=&  \mathbb{E}\int_0^T\int_{\mathcal{D}}  \lambda(u) \cdot (v - u) \; dxdt + 
  \mathbb{E}\int_0^T\int_{\mathcal{D}}  \beta u\cdot (v-u)  \;dxdt, \nonumber\\
   &=& \mathbb{E}\int_0^T\int_{\mathcal{D}} \left( \lambda(u) +   \beta u \right) \cdot (v - u) \; dxdt \geq 0, \label{condj2}
\end{eqnarray}
yields the desired result,
thereby completing the proof of the theorem.
\end{proof}
\begin{remark}
 Note that if we set $\mathcal{U}_{ad}=\mathcal{U},$ then (\ref{condj2}) implies that
 \[
  \lambda(u) +   \beta u =0. 
 \]
\end{remark}
\begin{remark}
 It is pertinent to point out here that  the papers \cite{CQ14, HLM11}
 prove the existence and uniqueness of the solution of SOCPs in the particular case of the steady-state diffusion equation 
 constraint, with the operator $A$  given by (\ref{diffeq}). However, ours is a generalization of 
 their result to the parabolic case, with 
 operator $A$ satisfying the assumptions $(a)$ and $(b)$ in Section \ref{probst}. Moreover,
 unlike this work,  \cite{CQ14, HLM11}  
 do not consider the inclusion of the standard deviation of the state variable
 (or any other risk measure)
 in the cost functional.
 \end{remark}

\section*{Acknowledgement}
 The work was performed while Akwum Onwunta was at the Max
Planck Institute for Dynamics of Complex Technical Systems, Magdeburg.

\bibliographystyle{siam}
\bibliography{akwumref}

\end{document}